\documentclass{amsart}
\usepackage{graphicx}
\usepackage{amssymb}
\usepackage{amsfonts}
\swapnumbers
\newtheorem{theorem}{Theorem}[section]
\newtheorem{lemma}[theorem]{Lemma}

\newtheorem{proposition}[theorem]{Proposition}
\theoremstyle{definition}

\newtheorem{assumption}[theorem]{Assumption}
\newtheorem{remark}[theorem]{Remark}

\numberwithin{equation}{section}
 \theoremstyle{plain}    
 
 \numberwithin{equation}{section} 
 \numberwithin{figure}{section} 
 \theoremstyle{plain}    
 \theoremstyle{plain}    
 \theoremstyle{remark}    
 \newtheorem*{acknowledgement*}{Acknowledgement} 

\newcommand{\cF}{{\mathcal F}}
\newcommand{\cG}{{\mathcal G}}
\newcommand{\cH}{{\mathcal H}}

\newcommand{\te}{{\theta}}

\newcommand{\Om}{{\Omega}}
\newcommand{\om}{{\omega}}
\newcommand{\ve}{{\varepsilon}}
\newcommand{\del}{{\delta}}
\newcommand{\Del}{{\Delta}}

\newcommand{\Gam}{{\Gamma}}

\newcommand{\vp}{{\varpi}}
\newcommand{\io}{{\iota}}
\newcommand{\up}{{\upsilon}}
\newcommand{\Up}{{\Upsilon}}

\newcommand{\sig}{{\sigma}}
\newcommand{\al}{{\alpha}}
\newcommand{\be}{{\beta}}
\newcommand{\ka}{{\kappa}}

\newcommand{\Sq}{{\large\square}}

\newcommand{\bbN}{{\mathbb N}}

\newcommand{\bbR}{{\mathbb R}}

\newcommand{\bbZ}{{\mathbb Z}}

\newcommand{\bt}{{\bf t}}
\newcommand{\bd}{{\bf d}}
\newcommand{\bfe}{{\bf e}}

\begin{document}
\title[]{A nonconventional invariance principle\\
 for random fields}%
 \vskip 0.1cm 
 \author{ Yuri Kifer\\
 \vskip 0.1cm
Institute of Mathematics\\
The Hebrew University of Jerusalem}%
\email{kifer\@@math.huji.ac.il}
\address{Institute of Mathematics, Hebrew University, Jerusalem 91904,\linebreak
 Israel}

\thanks{ }
\dedicatory{  }
 \date{\today}
\begin{abstract}\noindent
In \cite{KV} we obtained a nonconventional invariance principle (functional 
central limit theorem) for sufficiently fast mixing stochastic processes with
discrete and continuous time. In this paper we derive a nonconventional 
invariance principle for sufficiently well mixing random fields. 
\end{abstract}
\maketitle
\markboth{Yu.Kifer}{Random fields} 
\renewcommand{\theequation}{\arabic{section}.\arabic{equation}}
\pagenumbering{arabic}

\section{Introduction}\label{sec1}\setcounter{equation}{0}

Nonconventional ergodic theorems (see \cite{Fu}) known also after \cite{Be}
as polinomial ergodic theorems studied the limits
 of expressions having the 
form $1/N\sum_{n=1}^NT^{q_1(n)}f_1\cdots T^{q_\ell (n)}f_\ell$ where $T$ is a
weakly mixing measure preserving transformation, $f_i$'s are bounded measurable
functions and $q_i$'s are polynomials taking on integer values on the integers.
Originally, these results were motivated by applications to multiple recurrence
for dynamical systems taking functions $f_i$ being indicators of some measurable
sets. Later such results were extended to the case when $q_i$'s are
polinomials on $\bbZ^\nu$ (see \cite{Le}) and to some $\bbZ^\nu$ actions
(see \cite{Au}).

Using the language of probability this kind of results may be called 
nonconventional laws of large numbers and as a natural follow up we arrived at
the invariance principle (functional central limit theorem) in \cite{KV} 
showing convergence in distribution to Gaussian processes for expressions
 of the form
\begin{equation}\label{1.1}
1/\sqrt{N}\sum_{0\leq n\leq Nt}\big(F\big(X(q_1(n)),..., X(q_\ell(n))\big)
-\bar F\big)
\end{equation}
where $X(n),n\geq 0$ is a sufficiently fast $\al,\rho$ or $\psi$-mixing 
vector valued process with some moment conditions and stationarity properties,
$F$ is a continuous function with polinomial growth and certain regularity
properties, $\bar F=\int Fd(\mu\times\cdots\times\mu)$, $\mu$ is the 
distribution of each $X(n)$, $q_j(n)=jn,\, j\leq k\leq\ell$ and $q_j,j=k+1,
...,\ell$ are positive
functions taking on integer values on integers with some growth conditions 
which are satisfied, for instance, when $q_i$'s are polynomials of growing 
degrees. 

The goal of this paper is to prove an invariance principle type result when
$n\in\bbZ^\nu$ is multidimensional. This can be done either by considering
functions $q_i:\bbZ^\nu\to\bbZ_+$ with $X(n),n\geq 0$ being again a vector
valued stochastic process or, more generally, considering maps
$q_i:\bbZ^\nu\to\bbZ^\nu$ taking now $X(n),n\in\bbZ^\nu$ to be a vector 
valued random field which will be our setup in this paper. Namely, for
$t=(t_1,...,t_\nu)\in [0,1]^\nu$ and a positive integer $N$ we  consider 
expressions of the form
\begin{equation}\label{1.2}
\xi_N(t)=N^{-\nu/2}\sum_{n=(n_1,...,n_\nu):0\leq n_i\leq Nt_i\,\forall i}\big(
F\big(X(q_1(n)),..., X(q_\ell(n))\big)-\bar F\big)
\end{equation}
where $X(n),n\in\bbZ^\nu$ is
a sufficiently well mixing vector valued random field, with some moment 
conditions and stationarity properties, $F$ and $\bar F$ are similar to above,
$q_j(n)=jn,\, j\leq k$ and $q_i:\bbZ^\nu\to\bbZ^\nu,\, i=k+1,...,\ell$ map
$\bbZ_+^\nu=\{ n=(n_1,...,n_\nu)\in\bbZ^\nu:\, n_i\geq 0,\, i=1,...,\nu\}$
into itself. Assuming some growth conditions of $|q_i|,\, i>k$ in $|n|$ we 
will show that the random field $\xi_N(t)$ converges in distribution to a
Gaussian random field on $[0,1]^\nu$.

In \cite{KV} we were able to obtain the latter result for one dimensional $n$
relying on martingale approximations and martingale limit theorems but for 
random fields this machinery is not readily available. Still, we are able
to combine some of mixingale technique from \cite{ML1} and \cite{ML2} together 
with an appropriate grouping of summands in (\ref{1.2}) in order to obtain both 
convergence of finite dimensional distributions and the tightness of infinite 
dimensional ones.
Other known methods which work successfully when proving limit theorems for 
random fields (see, for instance, \cite{Bo}, \cite{BS}, \cite{BS2} and 
\cite{Su}  ) rely
one way or another on characteristic functions (or other devices based on
weak dependence) which are hard to deal with in the nonconventional setup 
as demonstrated in \cite{Ki} in view of the strong dependence of the summands
 in (\ref{1.2}) on the far away members of the random field. For specific
 lattice models with sufficiently good mixing properties to fit our setup
 we refer the reader to \cite{Al} and references there.

 \section{Preliminaries and main results}\label{sec2}\setcounter{equation}{0}

Our setup consists of a  $\wp$-dimensional vector random field
$\{X(n),\, n\in\bbZ^\nu,\, X(n)\in\bbR^\wp\}$ on a probability space $(\Om,\cF,P)$ and of a family of
$\sig$-algebras $\cF_{\Gam}\subset\cF,\, \Gam\subset\bbZ^\nu$
such that $\cF_{\Gam}\subset\cF_{\Del}$ if $\Gam\subset\Del\subset\bbZ^\nu$.
It is often convenient to measure the dependence between two sub
$\sig$-algebras $\cG,\cH\subset\cF$ via the quantities
\begin{equation}\label{2.1}
\varpi_{q,p}(\cG,\cH)=\sup\{\| E\big (g|\cG\big)-Eg\|_p:\, g\,\,\mbox{is}\,\,
\cH-\mbox{measurable and}\,\,\| g\|_q\leq 1\}
\end{equation}
where the supremum is taken over real functions and $\|\cdot\|_r$ is the
$L^r(\Om,\cF,P)$-norm. Then more familiar $\al,\rho,\phi$ and $\psi$-mixing 
(dependence) coefficients can be expressed via the formulas (see \cite{Bra},
Ch. 4 ),
\begin{eqnarray*}
&\al(\cG,\cH)=\frac 14\varpi_{\infty,1}(\cG,\cH),\,\,\rho(\cG,\cH)=\varpi_{2,2}
(\cG,\cH)\\
&\phi(\cG,\cH)=\frac 12\varpi_{\infty,\infty}(\cG,\cH)\,\,\mbox{and}\,\,
\psi(\cG,\cH)=\varpi_{1,\infty}(\cG,\cH).
\end{eqnarray*}
We set also
\begin{equation}\label{2.2}
\varpi_{q,p}(r)=\sup_{\Gam,\Del:\,dist(\Gam,\Del)\geq r}\big(
|\Gam\cup\Del|^{-1}\varpi_{q,p}(\cF_{\Gam},\cF_{\Del})\big)
\end{equation}
where $\Gam$ and $\Del$ are finite nonempty subsets of $\bbZ^\nu$,
$dist(\Gam,\Del)=\inf_{n\in\Gam,\tilde n\in\Del}|n-\tilde n|$ and
we write $|\Gam|$ for cardinality of a set $\Gam$ while, as usual, for
 numbers or vectors $|\cdot |$ will denote their absolute values or lengths.
As shown in \cite{Do} imposing decay conditions on dependence coefficients
which do not take into account sizes of sets $\Gam$ and $\Del$ as in (\ref{2.2})
would exclude from our setup simple examples of Gibbs random fields. Define  
also
\[
\al(l)=\frac{1}{4}\varpi_{\infty,1}(l),\,\rho(l)=\varpi_{2,2}(l),\,
\phi(l)=\varpi_{\infty,\infty}(l)\,\,\mbox{and}\,\,\ \psi(l)=\varpi_{1,
\infty}(l).
\]
Our setup includes also conditions on the  approximation rate
\begin{equation}\label{2.3}
\beta_p(r)=\sup_{n\in\bbZ^\nu}\|X(n)-E\big(X(n)|\cF_{U_r(n)}\big)\|_p
\end{equation}
where $U_r(n)=\{ \tilde n\in\bbZ^\nu:\, |n-\tilde n|\leq r\}$ is the 
$r$-neghborhood on $n$ in $\bbZ^\nu$. Furthermore, we do not require 
stationarity of the random field $X(n), n\in\bbZ^\nu$ assuming only that the
distribution of $X(n)$ does not depend on $n$ and the joint distribution of 
$\{X(n), X(n')\}$  depends only on $n-n'$ which we write for further references
by 
\begin{equation}\label{2.4}
X(n)\stackrel {d}{\sim}\mu\,\,\mbox{and}\,\,
(X(n),X(n'))\stackrel {d}{\sim}\mu_{n-n'}\,\,\mbox{for all}\,\,n,n'
\end{equation}
where $Y\stackrel {d}{\sim}Z$ means that $Y$ and $Z$ have the same distribution.

Next, let $F= F(x_1,...,x_\ell),\, x_j\in\bbR^{\wp}$ be a function on 
$\bbR^{\wp\ell}$ such that for some $\iota,K>0,\ka\in (0,1]$ and all 
$x_i,y_i\in\bbR^{\wp}, i=1,...,\ell$, we have
\begin{equation}\label{2.5}
|F(x_1,...,x_\ell)-F(y_1,...,y_\ell)|\leq K\big(1+\sum^\ell_{j=1}|x_j|^\iota+
\sum^\ell_{j=1} |y_j|^\iota \big)\sum^\ell_{j=1}|x_j-y_j|^\ka
\end{equation}
and 
\begin{equation}\label{2.6} 
|F(x_1,...,x_\ell)|\leq K\big( 1+\sum^\ell_{j=1}|x_j|^{\iota} \big).
\end{equation}
Our assumptions on $F$ enable us to include, for instance,
products $F(x_1,...,x_\ell)=x_{11}x_{22}\cdots x_{\ell\ell}$, where 
$x_i=(x_{i1},...,x_{i\ell})\in\bbR^\ell$, which is sometimes useful.
To simplify formulas we assume a centering condition
\begin{equation}\label{2.7}
\bar F=\int F(x_1,...,x_\ell)\,d\mu(x_1)\cdots d\mu(x_\ell)=0
\end{equation}
which is not really a restriction since we always can replace $F$ by 
$F-\bar F$.

Our goal is to prove an invariance principle (functional central limit theorem)
for $\xi_N(t),\, t\in [0,1]^\nu$ defined by (\ref{1.2}) where $q_j(n)=jn$ if
$j=1,2,...,k\leq\ell$ for some given positive integers $k,\ell$ and if $k<\ell$ then
 $q_j:\,\bbZ^\nu\to\bbZ^\nu,\, j=k+1,...,\ell$ satisfy the
conditions below. Before we formulate them observe that already the case $k=\ell$
is of major interest and we add $q_j's$ with $j>k$ which grow faster than linearly
mostly for the sake of completeness which under appropriate assumptions does 
not cause substantial problems.
We assume that $|q_1(n)|< |q_2(n)| <\cdots <|q_\ell(n)|$ whenever $|n|\ne 0$ 
and $n\in\bbZ_+^\nu=\{ m=(m_1,...,m_\nu)\in\bbZ^\nu:\, 
 m_i\geq 0\,\forall i\}$. Furthermore, we assume that for $k+1\le i \le \ell$,
 \begin{equation}\label{2.8-}
 \inf_{|\tilde n|>|n|}\big(|q_i(\tilde n)|-|q_i(n)|\big)(|\tilde n|-|n|)^{-1}
 >0,
 \end{equation}
\begin{equation}\label{2.8}
\lim_{|n|\to\infty}\inf_{\tilde n:\tilde n\ne n}\big(|q_i(\tilde n)-q_i(n)|-
|\tilde n-n|\big)=\infty,
\end{equation}
\begin{equation}\label{2.8+}
q_i(n)\ne q_i(\tilde n)\,\,\mbox{if}\,\, n\ne\tilde n,\,\,
\lim_{|n|\to\infty}\min_{l<i}\big(|q_i(n)|-|q_l(n)|-|n|\big)=\infty
\end{equation}
and for any $\ve > 0$,
\begin{equation}\label{2.9}
\lim_{|n|\to\infty}\inf_{\tilde n:\, |\tilde n|\geq\ve |n|}\min_{l<i}\big(
|q_{i}(\tilde n)|-|q_l(n)|-|\tilde n-n|\big)=\infty.
\end{equation}
For each $\theta>0$ set
\begin{equation}\label{2.10}
\gamma_\theta^\theta = \|X\|_\theta^\theta= E|X(n)|^\theta  =
\int |x|^\theta d\mu .
\end{equation}
Our main result relies on

\begin{assumption}\label{ass2.1} With $d=(\ell-1)\wp$ there exist $p,q\ge 1$,
 $m\geq 4$ and $\delta >0$ with   $ \delta \le \ka$, $ p\ka>d$ satisfying 
\begin{equation}\label{2.11}
\sum_{l=0}^\infty l^{5\nu}\varpi_{q,p}(l)=\te(p,q)<\infty,
\end{equation}
\begin{equation}\label{2.12}
\sum_{r=0}^\infty r^{5\nu}\beta^\del_q( r)<\infty,
\end{equation}
\begin{equation}\label{2.13}
\gamma_{m}<\infty, \,\gamma_{2q\iota}<\infty\,\,\mbox{with}\,\,
 \frac{1}{2}\ge \frac{1}{p}+\frac{\iota+1}{m}+\frac{\delta}{q}.
 \end{equation}
 \end{assumption}
 
 In order to give a detailed statement of our main result as well as for its
proof it will be essential to represent the function $F= F(x_1,x_2,\ldots,
x_\ell)$ in the form 
\begin{equation}\label{2.14}
F=F_1(x_1)+\cdots+F_\ell(x_1, x_2,\ldots, x_\ell)
\end{equation}
where for $i<\ell$,
\begin{eqnarray}\label{2.15}
&F_i(x_1,\ldots, x_i)=\int F(x_1,x_2, \ldots, x_\ell)\ d\mu (x_{i+1})\cdots 
d\mu(x_\ell)\\
&\quad -\int F(x_1,x_2, \ldots, x_\ell) \,d\mu (x_i)\cdots d\mu(x_\ell)\nonumber
\end{eqnarray}
and
\[
F_\ell(x_1,x_2, \ldots, x_\ell)=F(x_1,x_2, \ldots, x_\ell) -\int F(x_1,x_2, 
\ldots, x_\ell)\, d\mu(x_\ell)
\]
which ensures, in particular, that
\begin{equation}\label{2.16}
\int F_i(x_1, x_2,\ldots,x_{i-1}, x_i)\,d\mu(x_i)\equiv 0 \quad\forall 
\quad x_1, x_2,\ldots, x_{i-1}.
\end{equation}
We write $t=(t_1,...,t_\nu)\geq s=(s_1,...,s_\nu)$ if $t_i\geq s_i$ for all $i$
and for such $s,t\in[0,1]^\nu$ we set $\Del_N(s,t)=\{ n=(n_1,...,n_\nu)\in
\bbZ^\nu:\, Ns_i\leq n_i\leq Nt_i\,\forall i\}$ and $\Del_N(t)=\Del_N(0,t)$.
These together with (\ref{2.14})--(\ref{2.16}) enable us to represent 
$\xi_N(t)$ given by (\ref{1.2}) in the form
\begin{equation}\label{2.17}
\xi_N(t)=\sum_{i=1}^k \xi_{i,N}(it)+\sum_{i=k+1}^\ell \xi_{i,N}(t)
\end{equation}
where for $1\leq i\leq k$,
\begin{equation}\label{2.18}
\xi_{i,N}(t)= N^{-\nu/2}\sum_{n\in\Del_N(t/i)} F_i(X(n), X(2n),\ldots, X(in))
\end{equation}
and for $i\ge k+1$,
\begin{align}\label{2.19}
\xi_{i,N}(t)=N^{-\nu/2}\sum_{n\in\Del_N(t)}  F_i(X(q_1(n)),\ldots,
 X(q_i(n))).
\end{align}
Next, we define a matrix $(D_{i,j},\, 1\leq i,j\leq k)$ which appears in the limiting
covariances formula in our main result below. For any $i,j\leq k$ set
\begin{equation}\label{2.19+}
D_{i,j}=\frac {\up}{ij}\sum_{u\in\bbZ^\nu} c_{i,j}(u)
\end{equation}
where $\up$ is the greatest common divisor of $i$ and $j$, $c_{i,j}(u)=0$ if $\up$
does not divide all components of $u\in\bbZ^\nu$ and for $i'=i/\up,\, j'=j/\up$,
\begin{eqnarray}\label{2.19++}
&c_{i,j}(\up u)=\int F_i(x_1,...,x_i)F_j(y_1,...,y_j)\prod_{\al=1}^\up 
d\mu_{\al u}(x_{\al i'},y_{\al j'})\\
&\prod_{\sigma\notin \{i',2i',\ldots,\up i'\}} d\mu(x_\sigma)
\prod_{\sigma'\notin \{j',2j',\ldots,\up j'\}}d\mu(y_{\sigma'})\nonumber
\end{eqnarray}
with $\mu_0$ being the diagonal measure, i.e. $\int f(x,y)d\mu_0(x,y)=\int f(x,x)d\mu(x)$.

\begin{theorem}\label{thm2.2}
Suppose that the conditions (\ref{2.4})--(\ref{2.9}) and Assumption \ref{ass2.1} 
hold true then each random field 
$\xi_{i,N}(t),\, i=1,2,...,\ell$ converges in distribution as $N\to\infty$ 
to a Gaussian random field $\eta_i(t)$. Moreover, 
$(\eta_1(t),\eta_2(t),...,\eta_\ell(t))$ is an $\ell$-dimensional Gaussian
random field such that $\eta_i(t),\, i\leq k$ have covariances 
\[
E\eta_i(s)\eta_j(t)=D_{i,j}\prod_{l=1}^\nu\min(s_l,t_l),\,\,i,j\leq k
\]
with matrix $D_{i,j}$ defined by (\ref{2.19+}) while the random
fields $\eta_i(t),\, i\geq k+1$ are independent of each other and of 
$\eta_j$'s with $j\leq k$ and have variances
\[
E|\eta_i(t)|^2=\int|F_i(x_1,x_2,...,x_i)|^2d\mu(x_1)d\mu(x_2)\cdots d\mu(x_i),\,\,
i\geq k+1.
\]
Finally, $\xi_N(t)$ converges in distribution to a Gaussian random field
$\xi(t)$ which can be represented in the form
\begin{equation}\label{2.20}
\xi(t)=\sum_{i=1}^k\eta_i(it)+\sum_{i=k+1}^\ell \eta_i(t).
\end{equation}
\end{theorem}

In order to understand our assumptions observe that $\varpi_{q,p}$
 is clearly non-increasing in $q$ and non-decreasing in $p$. Hence,
 for any pair $p,q\geq 1$,
 \[
 \varpi_{q,p}(n)\leq\psi(n).
 \]
 Furthermore, by the real version of the Riesz--Thorin interpolation 
 theorem or the Riesz convexity theorem (see \cite{Ga}, Section 9.3
 and \cite{DS}, Section VI.10.11) whenever $\theta\in[0,1],\, 1\leq
 p_0,p_1,q_0,q_1\leq\infty$ and
 \[
 \frac 1p=\frac {1-\theta}{p_0}+\frac \theta{p_1},\,\,\frac 1q=\frac
 {1-\theta}{q_0}+\frac \theta{q_1}
 \]
 then
 \begin{equation}\label{2.21}
\varpi_{q,p}(n)\le 2(\varpi_{q_0,p_0}(n))^{1-\theta}
(\varpi_{q_1,p_1}(n))^\theta.
\end{equation}
In particular,  using the obvious bound $\varpi_{q_1,p_1}\leq 2$
valid for any $q_1\geq p_1$ we obtain from (\ref{2.21}) for pairs
$(\infty,1)$, $(2,2)$ and $(\infty,\infty)$ that for all $q\geq p\geq 1$,
 \begin{eqnarray}\label{2.22}
&\varpi_{q,p}(n)\le (2\alpha(n))^{\frac{1}{p}-\frac{1}{q}},\,
 \varpi_{q,p}(n)\le 2^{1+\frac 1p-\frac 1q}(\rho(n))^{1-\frac 1p+\frac 1q}\\
&\mbox{and}\,\,\varpi_{q,p}(n)\le 2^{1+\frac 1p}(\phi(n))^{1-\frac 1p}.
\nonumber\end{eqnarray}
We observe also that by the H\" older inequality for $q\geq p\geq 1$
and $\alpha\in(0,p/q)$,
\begin{equation}\label{2.23}
\beta_q(r)\le 2^{1-\alpha}  [\beta_p(r)]^\alpha \gamma^{1-\al}_{\frac{pq(1-\al)}
{p-q\al}}
\end{equation}
with $\gamma_\theta$ defined in (\ref{2.10}). Thus, we can formulate 
Assumption \ref{ass2.1} in terms of more familiar $\alpha,\,\rho,\,\phi,$
and $\psi$--mixing coefficients and with various moment conditions. It 
follows also from (\ref{2.21}) that if $\varpi_{q,p}(n)\to 0$ as $n\to\infty$
for some $q\geq p\geq 1$ then
\begin{equation}\label{2.24}
\varpi_{q,p}(n)\to 0\,\,\mbox{as}\,\, n\to\infty\,\,\mbox{for all}\,\,
q> p\geq 1,
\end{equation}
and so (\ref{2.24}) holds true under Assumption \ref{ass2.1}.

In order to prove Theorem \ref{thm2.2} we will represent $\xi_{i,N}(t)$ 
in the form $\sum_{1\leq l\leq N}Z_{t,N}(l)$ where now $l$ is one dimensional
which together with estimates of the next section will enable us to apply 
central limit theorems for mixingale arrays (see \cite{ML1} and \cite{ML2}). 
This will lead to Gaussian one dimensional distributions in the limit but 
combining this with a kind of the Cram\' er-Wold argument, covariances 
computation in Section \ref{sec4} and tightness estimates of Section \ref{sec5}
will yield appropriate Gaussian random fields as asserted in the theorem.
Recall (see \cite{KV}), that already in the one parameter case $\nu=1$ the
process $\xi(t)$, in general, does not have independent increments so also
in the random field case $\xi(t)$, in general, is not a multiparameter
Brownian motion.

\begin{remark}\label{rem2.3} As a part of tightness estimates of Section 
\ref{sec5} we will see that $\sup_{N\geq 1,\, t\in[0,1]^\nu}E|\xi_{i,N}(t)|^4
\leq C<\infty$. Hence, applying the Borel--Cantelli lemma we obtain as a 
byproduct that if $S_{i,N}=N^{\nu/2}\xi_{i,N}(t)$ and $S_{N}=N^{\nu/2}
\sum_{i=1}^\ell\xi_{i,N}(t)$ then with probability one
\[
\lim_{N\to\infty}\frac 1{N^\nu}S_{i,N}(t)=0\,\,\mbox{for each $i$, and so}\,\,
\lim_{N\to\infty}\frac 1{N^\nu}S_N(t)=0.
\]
Still, we observe that this strong law of large numbers can be obtained under
more general circumstances here since, in particular, we do not need for it
convergence of covariances derived in Section \ref{sec4} which requires, for
instance, more specific assumptions on $q_j$'s.
\end{remark}

\section{Blocks and mixingale type estimates}
\label{sec3}\setcounter{equation}{0}

We rely on the following result which appears as Corollary 3.6 in \cite{KV}. 

\begin{proposition}\label{prop3.1}
Let $\cG$ and $\cH$ be $\sig$-subalgebras on a probability space $(\Om,\cF,P)$,
$X$ and $Y$ be $d$-dimensional random vectors and $f=f(x,\om),\, x\in\bbR^d$ be a 
collection of random variables measurable with respect to $\cH$ and satisfying
\begin{equation}\label{3.1}
\|f( x,\omega)-f( y,\omega)\|_{q}\le C (1+|x|^\iota + |y|^\iota)|x-y|^\ka
\,\,\mbox{and}\,\,\|f(x,\omega)\|_{q}\le C (1+|x|^\iota)
\end{equation}
where  $q\geq 1$. 
Set $g(x)=Ef(x,\om)$. Then 
\begin{equation}\label{3.2}
\| E(f(X,\cdot)|\cG)-g(X)\|_\up\leq c(1+\| X\|^{\io+2}
_{b(\io+2)})(\vp_{q,p}(\cG,\cH)+\| X-E(X|\cG)\|^\del_{q})
\end{equation} 
provided $\ka-\frac dp>0$, $\frac 1\up\geq\frac 1p+
\frac 1{b}+\frac {\del}q$ with $c=c(C,\io,\ka,\del,p,q,\up,d)>0$ 
depending only on parameters in brackets. Moreover, let
$x=(w,z)$ and $X=(W,Z)$, where $W$ and $Z$ are $d_1$ and $d-d_1$-dimensional
random vectors, respectively, and let $f(x,\om)=f(w,z,\om)$ satisfy (\ref{3.1})
in $x=(w,z)$. Set $\tilde g(w)=Ef(w,Z(\om),\om)$. Then
\begin{eqnarray}\label{3.3}
&\| E(f(W,Z,\cdot)|\cG)-\tilde g(W)\|_\up\leq c(1+\| X\|^{\io+2}
_{b(\io+2)})\\
&\times\big(\vp_{q,p}(\cG,\cH)
+\| W-E(W|\cG)\|^\del_{q}+\| Z-E(Z|\cH)\|^\del_{q}\big).\nonumber
\end{eqnarray}
\end{proposition}

We will use the following notations
\begin{eqnarray}\label{3.5}
&\quad F_{i,n,r}=F_{i,n,r}(x_1,x_2,\ldots, x_{i-1},\omega)
=E\big(F_i(x_1,x_2,\ldots, x_{i-1},X(n))|\cF_{U_r(n)}\big),\\
&Y_i(q_i(n))=F_i(X(q_1(n)),\ldots, X(q_i(n)))\,\,\mbox{and}\,\,
Y_i(m)=0\,\,\mbox{if}\,\, m\ne q_i(n)\,\,\mbox{for any}\,\, n,\nonumber\\
&X_r(n)=E(X(n)|\cF_{U_r(n)}),\, Y_{i,r}(q_i(n))=F_{i,q_i(n),r} (X_r(q_1(n)),
\nonumber\\
&\ldots, X_r(q_{i-1}(n)),\omega)\,\,
\mbox{and}\,\, Y_{i,r}(m)=0\,\,\mbox{if}\,\, m\ne q_i(n)\,\,\mbox{for any}\,
 n;\nonumber\\
&\bar Y_i(n)=Y_i(n)-EY_i(n),\,\bar Y_{i,r}(n)=Y_{i,r}(n)-EY_{i,r}(n).
\nonumber
\end{eqnarray}

For each $l\in\bbZ_+$ introduce cubes $\Sq(l)=\{ n=(n_1,...,n_\nu)\in\bbZ^\nu:
0\leq n_i\leq l\,\,\mbox{for}\,\, i=1,...,\nu\}$ and for $l<\tilde l$ we set 
also $\Up(l,\tilde l)=\Sq(\tilde l)\setminus\Sq(l)$. Fix some positive numbers
$\frac 5{11}(\tau+1)<3\eta<\tau<1$ and set $a(1)=0,b(1)=1$ and for 
$j>1$,
\begin{equation}\label{3.6}
a(j)=b(j-1)+[(j-1)^{2\eta}],\, b(j)=a(j)+[j^\tau]\,\,\mbox{and}\,\, r(j)=[j^\eta].
\end{equation}
Set $\Del^{(i)}_N(t)=\Del_N(t/i)$ if $i\leq k$ and $\Del_N^{(i)}(t)=
\Del_N(t)$ if $i\geq k+1$. We define now 
\begin{eqnarray}\label{3.7}
&V_{i,t,N}(l)=\sum_{n\in\Del^{(i)}_N(t)\cap\Up(a(l),b(l))}Y_{i,r(l)}(q_i(n))\\
&\mbox{and}\,\,W_{i,t,N}(l)=\sum_{n\in\Del^{(i)}_N(t)\cap\Up(b(l),a(l+1))}
Y_{i,r(l)}(q_i(n))\nonumber
\end{eqnarray}
The sets $\Up(b(l),a(l+1))$ will play the role of gaps between 
$\Up(a(l),b(l))$ and $\Up(a(l+1),b(l+1))$ and we will see that the random 
variables $W_{i,t,N}(l)$ can be disregarded for our purposes while dealing 
with the random variables $V_{i,t,N}(l)$ we will take advantage of our mixing 
conditions in order to show that their centered versions $\bar V_{i,t,N}(l)=
V_{i,t,N}(l)-EV_{i,t,N}(l)$ satify mixingale estimates (see \cite{ML1} and
\cite{ML2}) with respect to the nested family of $\sig$-algebras $\cG^{(i)}_l=
\cF_{\Gam_i(l)},\, l=0,1,2,...$ where 
\[
\Gam_i(l)=\{ n\in\bbZ^\nu_+:\,\mbox{dist}\big(n,\cup_{j\leq i}q_j(\Sq(b(l)))
\big)\leq r(l)\}
\]
and we take $\cG^{(i)}_l$ to be the trivial $\sig$-algebra $\{\emptyset,\Om\}$
for $l<0$.

Namely, for any $u\in\bbN$ we have
\begin{eqnarray}\label{3.8}
&\| E\big(\bar V_{i,t,N}(l)|\cG^{(i)}_{l-u}\big)\|_2\leq\sum_{n\in\Del_N(t)\cap
\Up(a(l),b(l))}\| E(\bar Y_{i,r(l)}(q_i(n))|\cG^{(i)}_{l-u})\|_2\\
&\leq |\Up(a(l),b(l))|\max_{n\in\Up(a(l),b(l))}\| E(\bar Y_{i,r(l)}(q_i(n))|
\cG^{(i)}_{l-u})\|_2
\nonumber\end{eqnarray}
where $|A|$ for a set $A$ denotes its cardinality. Next, for $u>l$,
\begin{equation}\label{3.9}
E(\bar Y_{i,r(l)}(q_i(n))|\cG^{(i)}_{l-u})=0
\end{equation}
while for all $u\geq 2$ by the Cauchy inequality and the
contraction property of conditional expectations 
\begin{equation}\label{3.10}
\| E(\bar Y_{i,r(l)}(q_i(n))|\cG^{(i)}_{l-u})\|_2\leq 2\| E(Y_{i,r(l)}
(q_i(n))|\cG^{(i)}_{l-2})\|_2.
\end{equation}
Observe that if $n\in\Up(a(l),b(l))$ then $X=\big(X_{r(l)}(q_1(n)),...,
X_{r(l)}(q_{i-1}(n))\big)$ is $\cG_l^{(i-1)}$-measurable and for large $l$
we obtain also by the definition of $q_i$ for $i\leq k$ and by (\ref{2.8-})
and (\ref{2.9}) for $i>k$ that
\begin{equation}\label{3.11}
\mbox{dist}\big(\Gam_i(l-2)\cup\Gam_{i-1}(l),q_i(n)\big)\geq (l-1)^\tau
\end{equation}
taking into account that $a(l)>\frac 1{1+\tau}(l-1)^{1+\tau}-(l-1)$.
We can write also that
\begin{equation}\label{3.11+}
b(l)\leq\frac 2{1+\tau}(l+1)^{1+\tau},\, |\Sq(l)|=(l+1)^\nu\,\,\mbox{and}\,\,
|\Up(\tilde l,l)|\leq\nu(l-\tilde l)(l+1)^{\nu-1}\,\,\mbox{if}\,\,   l>\tilde l.
\end{equation}
Thus, applying Proposition \ref{prop3.1} to the right hand side of 
(\ref{3.10}) with $\cG=\cF_{\Gam_i(l-2)\cup\Gam_{i-1}(l)}$ and $\cH=
\cF_{U_{r(l)}(q_i(n))}$ together with (\ref{3.11}), (\ref{3.11+}),
Assumption \ref{ass2.1} and
the contraction property of conditional expectations we obtain that
\begin{eqnarray}\label{3.12}
&\| E(Y_{i,r(l)}(q_i(n))|\cG^{(i)}_{l-2})\|_2\leq \| E(Y_{i,r(l)}(q_i(n))
|\cG)\|_2\\
&\leq C\vp_{p,q}(\cG,\cH)\leq\tilde Cl^{\nu(\tau+\eta+1)}\vp_{p,q}
((l-1)^\tau-l^\eta)\nonumber
\end{eqnarray}
for some $C,\tilde C>0$ independent of $n,l$ provided $p,q$ and $\del$ 
satisfy the conditions of Assumption \ref{ass2.1}. Collecting 
(\ref{3.7})--(\ref{3.10}), (\ref{3.11+}) and (\ref{3.12}) we obtain that 
for any $u\geq 2$,
\begin{equation}\label{3.13}
\| E(\bar V_{i,t,N}(l)|\cG^{(i)}_{l-u})\|_2\leq\tilde C2^{\nu(2+\tau)}
l^{\nu(2\tau+\eta+2)-1}\vp_{p,q}((l-1)^\tau-l^\eta).
\end{equation}

In order to incorporate (\ref{3.13}) into the setup of mixingale arrays from
\cite{ML2} we consider the triangular array $\hat V_{i,t,N}(l)=N^{-\nu/2}
\bar V_{i,t,N}(l),\, l=1,2,...,L(N);\, N=1,2,...$ where $L(N)=\min\{ j:\,
a(j+1)\geq N\}$. Observe that
\[
N\geq\sum_{1\leq j\leq L(N)-1}[j^\tau]\geq(1+\tau)^{-1}(L(N)-1)^{1+\tau},
\]
and so
\begin{equation}\label{3.18}
L(N)\leq (N(1+\tau))^{1/(1+\tau)}+1\leq 2N^{1/(1+\tau)}+1.
\end{equation}
Employing Lemma \ref{lem4.2} from the next section together with (\ref{2.13}),
(\ref{3.11+}) and (\ref{3.18}) we obtain that
\begin{equation}\label{3.13+}
\|V_{i,t,N}(l)\|^2_2\leq C|\Up(a(l),b(l))|\leq\frac {2C\nu}{1+\tau}l^\tau
(l+1)^{(1+\tau)(\nu-1)}\leq C_1N^{\nu-\frac 1{1+\tau}}
\end{equation}
for some $C,C_1>0$ independent of $l\leq L(N),N,i$ and $t$. Observe that by 
the choice of $\tau$ and $\eta$ we have that $\nu(2\tau+\eta+2)-1-5\nu\tau<-1$
 which together with (\ref{2.11}) and (\ref{3.13}) enables us
to write
\begin{equation}\label{3.13++}
\| E(\bar V_{i,t,N}(l)|\cG^{(i)}_{l-u})\|_2\leq C_2l^{-1}\leq C_2u^{-1}
\end{equation}
for some $C_2\geq 1$ independent of $l,N,i,t$ and $u=2,3,...,l$. Set
$\psi_j=C_2(\max(1,j))^{-1}$ for $j=0,1,2,...$ and $\sig_{l,N}=C_1N^{-\frac
1{2(1+\tau)}}$ for $l=1,2,...,L(N)$. Then (\ref{3.9}), (\ref{3.13+})
and (\ref{3.13++}) yield the first standard mixingale estimate
\begin{equation}\label{3.13+++}
\| E(\hat V_{i,t,N}(l)|\cG^{(i)}_{l-u})\|_2\leq\sig_{l,N}\psi_u
\end{equation}
 for all $u=0,1,2,...$
and the conditions imposed on $\psi_j$'s and $\sig_{l,N}$'s in \cite{ML2}
can be easily verified, as well. The second standard mixingale estimate 
(see \cite{ML2}) is trivial in our case since 
$\bar V_{i,t,N}(l)$ is $\cG^{(i)}_{l+u}$-measurable for any $u\geq 1$, and
so
\begin{equation}\label{3.14}
E(\bar V_{i,t,N}(l)|\cG^{(i)}_{l+u})-\bar V_{i,t,N}(l)=0.
\end{equation}

Next, we estimate contribution of small blocks (gaps) $W_{i,t,N}(j),\, 
j\geq 1$. Set 
\[
\hat\Gam_i(l)=\{ n\in\bbZ^\nu_+:\,\mbox{dist}\big(n,\cup_{j\leq i}q_j\big(
\Sq(a(l+1))\big)\big)\leq r(l)\},
\]
$\cG=\cF_{\hat\Gam_{i-1}(l)\cup\hat\Gam_i(l-2)}$ and 
$\cH=\cF_{U_{r(l)}(q_i(n))}$ where $n\in\Up(b(l),a(l+1))$. Observe that by
the properties of $q_j$'s for any such $n$ and large enough $l$,
\[
\mbox{dist}(\hat\Gam_{i-1}(l)\cup\hat\Gam_i (l-2),q_i(n))\geq l^\tau.
\]
Furthermore, if $j\leq l-2$ then $W_{i,t,N}(j)$ is 
$\cF_{\hat\Gam_i(l-2)}$-measurable, and so
employing Proposition \ref{prop3.1} with such $\cG$ and $\cH$ we obtain 
similarly to (\ref{3.12}) relying also on the Cauchy inequality that 
\begin{eqnarray}\label{3.15}
& |EW_{i,t,N}(j)Y_{i,r(l)}(q_i(n))|=\big\vert E\big(W_{i,t,N}(j)
E(Y_{i,r(l)}(q_i(n))|\cG)\big)\big\vert\\
&\leq C_1\| W_{i,t,N}(j)\|_2l^{\nu(\tau+\eta+1)}\vp_{q,p}(l^\tau-l^\eta)
\nonumber\end{eqnarray}
for some $C_1>0$ independent of $t,N,n,l$ and $j\leq l-2$. Hence, by 
(\ref{3.7}) and (\ref{3.11+}),
\begin{eqnarray}\label{3.16}
&|EW_{i,t,N}(j)W_{i,t,N}(l)|=|E\big(W_{i,t,N}(j)E(W_{i,t,N}(l)|\cG)\big)|\\
&\leq |\Up(b(l),a(l+1))|\max_{n\in\Up(b(l),a(l+1))}
|E\big(W_{i,t,N}(j)E(Y_{i,r(l)}(q_i(n)))|\cG)\big)|\nonumber\\
&\leq C_2\| W_{i,t,N}(j)\|_2
l^{\nu(2\tau+\eta+2)+2\eta-\tau-1}\vp_{q,p}(l^\tau-l^\eta)\big)\nonumber
\end{eqnarray}
for some $C_2>0$ independent of $t,N,n,l$ and $j\leq l-2$. 
Now we can write
\begin{eqnarray}\label{3.19}
&E\big(\sum_{j=0}^{L(N)}W_{i,t,N}(j)\big)^2\leq\sum_{j=0}^{L(N)}
\big(3EW^2_{i,t,N}(j)\\
&+2\sum_{l=j-2}^{L(N)}|EW_{i,t,N}(j)W_{i,t,N}(l)|\big)
\leq\sum_{j=0}^{L(N)}(3\| W_{i,t,N}(j)\|^2_2+C_3\| W_{i,t,N}(j)\|_2)
\nonumber\end{eqnarray}
where by (\ref{2.11}), (\ref{3.16}) and the choice of $\tau$ and $\eta$,
\[
C_3=C_2\sum_{1\leq l\leq L(N)}l^{\nu(2\tau+\eta+2)}\vp_{p,q}(l^\tau-l^\eta)
\leq C_2\sum_{l=1}^\infty l^{5\nu\tau}\vp_{p,q}(l^\tau-l^\eta)<\infty.
\]

Relying on (\ref{2.3}), (\ref{2.5}), (\ref{2.13}) and the H\" older inequality
 we can estimate the error of replacement of $Y_i(q_i(n))$ by its 
 $r(l)$-approximation $Y_{i,r(j)}(q_i(n))$ (see Lemma 3.12 in \cite{KV}),
 \begin{equation}\label{3.20}
 \| Y_i(q_i(n))-Y_{i,r(j)}(q_i(n))\|_2\leq C_4\be_q^\del(r(j))
 \end{equation}
 for some $C_4>0$ independent of $i,j$ and $n$. Now, set
 \[
 \zeta_{i,N}(t)=N^{-\nu/2}\sum_{1\leq l\leq L(N)}V_{i,t,N}(l).
 \]
 Then by (\ref{3.11+}), (\ref{3.19}) and (\ref{3.20}),
 \begin{eqnarray}\label{3.21}
 &\|\xi_{i,N}(t)-\zeta_{i,N}(t)\|_2\leq C_5N^{-\nu/2}\big(\sum_{1\leq l\leq
  L(N)}l^{\nu(\tau+1)-1}\be^\del_q([l^\eta])\\
  &+\big(\sum_{0\leq l\leq L(N)}(\| W_{i,t,N}(l)\|^2_2+
  \| W_{i,t,N}(l)\|_2)\big)^{1/2}\big).\nonumber
  \end{eqnarray}
  It follows from (\ref{2.13}) and Lemma \ref{lem4.2} of the next section that 
  \begin{equation}\label{3.22}
  \| W_{i,t,N}(l)\|^2_2=\mbox{O}(|\Up(b(l),a(l+1))|).
  \end{equation}
  By Assumption \ref{ass2.1} and the choice of $\eta$ and $\tau$ 
  we obtain from (\ref{2.12}), (\ref{3.11+}), (\ref{3.18}), (\ref{3.21}) and
  (\ref{3.22}) that 
  \begin{eqnarray}\label{3.23}
  &\|\xi_{i,N}(t)-\zeta_{i,N}(t)\|_2\leq C_6N^{-\nu/2}\big(
  \sum_{1\leq l\leq L(N)}l^{\nu(\tau+1-5\eta)-1}\\
  &+(2\sum_{1\leq l\leq L(N)}l^{(\nu-1)(\tau+1)+2\eta})^{1/2}\big)
  \nonumber\\
  &\leq C_7\big(N^{\nu(\frac 12-\frac {5\eta}{1+\tau})}+N^{\frac {2\eta-\tau}
  {2(1+\tau)}}\big)\to 0\,\,\mbox{as}\,\ N\to\infty,\nonumber
  \end{eqnarray}
  and so for each $t$ the limits in distribution as 
  $N\to\infty$ of $\xi_{i,N}(t)$ and of $\zeta_{i,N}(t)$ coincide (if they 
  exist).
  
  Observe that similarly to (\ref{3.8}), (\ref{3.12}) and (\ref{3.13})
  it follows from (\ref{2.11}), (\ref{3.11+}) together with the contraction
  property of conditional expectations that
  \begin{eqnarray}\label{3.24}
  &|E\zeta_{i,N}(t)|\leq N^{-\nu/2}\sum_{1\leq l\leq L(N)}|EV_{i,t,N}(l)|\\
  &\leq C_8N^{-\nu/2}\sum_{1\leq l\leq L(N)}l^{\nu(2\tau+\eta+2)-1}
  \vp_{p,q}((l-1)^\tau-l^\eta)\nonumber\\
  &\leq C_9N^{-\nu/2}\sum_{1\leq l\leq 
  L(N)}l^{\nu(-3\tau+\eta+2)-1}\leq C_{10}N^{-\frac {11}{18}\nu},\nonumber
  \end{eqnarray}
  for some $C_8,C_9,C_{10}>0$ independent of $N$, and so $|E\zeta_{i,N}(t)|
  \to 0$ as $N\to\infty$. Hence, the Gaussian limiting behavior of 
  $\zeta_{i,N}(t)-E\zeta_{i,N}(t)$ which we will derive via mixingale limit
  theorems yields the same Gaussian limiting behavior of $\zeta_{i,N}(t)$,
  and so in view of (\ref{3.23}) the same holds true for $\xi_{i,N}(t)$, as
  well.

  \section{Limiting covariances}\label{sec4}\setcounter{equation}{0}

The first step in our limiting covariances computations is the following 
estimate of
\begin{equation*}
b_{i,j}(m,n)=EY_i(q_i(m))Y_j(q_j(n)),\,\, m,n\in\bbZ^\nu_+
\end{equation*}
where $Y_i(q_i(n))$ was defined in (\ref{3.5}).

\begin{lemma}\label{lem4.1} (i) For $i,j=1,...,\ell$ and any $m,n\in\bbZ^\nu_+$
set 
\begin{eqnarray}\label{4.1}
&\hat s_{i,j}(m,n)=\min\big(\min_{1\leq l\leq j}|q_i(m)-q_l(n)|,\,|m|\big)\\
&\mbox{and}\,\, s_{i,j}(m,n)=\max(\hat s_{i,j}(m,n),\,\hat s_{j,i}(n,m)).
\nonumber\end{eqnarray}
Then for all $i,j\leq k$,
\begin{equation}\label{4.2}
s_{i,j}(m,n)\geq\frac 1{4k^2}|im-jn|.
\end{equation}
Furthermore, if $i\geq k+1$ then for any $\ve>0$ there exists $M_\ve>0$ 
such that if $\max(|m|,|n|)\geq M_\ve$ and $m\ne n$ then 
\begin{equation}\label{4.3}
s_{i,i}(m,n)\geq \min\big(|m-n|+\ve^{-1},\max(|m|,|n|)\big)\geq\frac 12|m-n|.
\end{equation}
(ii) There exists a function $h(l)\geq 0$ defined on integers such that 
$\sum_{l=1}^\infty l^{4\nu} h(l)<\infty$ and for any $i,j=1,2,...,\ell$ and
$l=0,1,2,...$,
\begin{equation}\label{4.4}
\sup_{m,n\in\bbZ^\nu_+:\, s_{i,j}(m,n)\geq l}|b_{i,j}(m,n)|\leq h(l).
\end{equation}
\end{lemma}
\begin{proof} (i) Let $i,j\leq k$ and set $u=|im-jn|$. Then $i|m|+j|n|\geq u$,
and so either $|m|\geq\frac u{2i}$ or $|n|\geq\frac u{2j}$. Suppose, for 
instance, that $|m|\geq\frac u{2i}$. If $\hat s_{i,j}(m,n)\geq\frac u{4k^2}$
then $s_{i,j}(m,n)\geq\frac u{4k^2}$ and we are done. So assume that 
$s_{i,j}(m,n)<\frac u{4k^2}$. Then $\min_{1\leq l\leq j}|im-ln|<\frac
u{4k^2}$, and so $|im-\hat ln|<\frac u{4k^2}$ for some $\hat l<j$. Then
$i|m|-\hat l|n|<\frac u{4k^2}$, whence 
\[
n>\frac 1{\hat l}(i|m|-\frac u{4k^2})\geq\frac u{2\hat l}-\frac u{4k^2}\geq
\frac u{4k}.
\]
Next, let $\min_{1\leq l<i}|jn-lm|=|jn-\tilde lm|,\,\tilde l<i$. Then
\begin{eqnarray*}
&|jn-\tilde lm|=|\frac j{\hat l}\hat ln-\frac j{\hat l}im+\frac j{\hat l}im-
\tilde lm|\geq(\frac j{\hat l}i-\tilde l)|m|-\frac j{\hat l}|\hat ln-im|\\
&\geq |m|-k\frac u{4k^2}\geq\frac u{2i}-\frac u{4k}\geq\frac u{4k}.
\end{eqnarray*}
It follows that $\hat s_{j,i}(n,m)\geq\frac u{4k}$, and so by above either
$\hat s_{ij}(m,n)\geq\frac u{4k^2}$ or $\hat s_{j,i}(n,m)\geq\frac u{4k}$, 
whence
$s_{i,j}(n,m)\geq\frac u{4k^2}$. The case $|n|\geq\frac u{2j}$ is dealt with
in the same way exchanging $i$ and $j$ as well as $m$ and $n$ in the above
argument.

In order to obtain (\ref{4.3}) we rely on the definition (\ref {4.1}) together
 with the assumptions (\ref{2.8}) and (\ref{2.9}).
 
 (ii) By (\ref{2.8+}) there exists $M$ such that $|q_i(m)-q_l(m)|\geq |m|$
 for all $l<i$ provided $|m|\geq M$. Hence, for such $m$,
 \[
 \min\big(\min_{1\leq l\leq j}|q_i(m)-q_l(n)|,\,\min_{1\leq l<i}|q_i(m)-
 q_l(m)|)\geq\hat s_{i,j}(m,n).
 \]
 Assume that $s_{i,j}(m,n)=\hat s_{i,j}(m,n)\geq 2r$ and set
\[
b^{(r)}_{i,j}(m,n)=EY_{i,r}(q_i(m))Y_{j,r}(q_j(n))
\]
where $Y_{i,r}(l)$ was defined in (\ref{3.5}). It follows from (\ref{2.3}),
(\ref{2.5}), (\ref{2.6})  together with the H\" older inequality (cf. Lemma
3.12 in \cite{KV}) that
\begin{equation}\label{4.5}
|b^{(r)}_{i,j}(m,n)-b_{i,j}(m,n)|\leq C(\be_q(r))^\del
\end{equation}
where a constant $C>0$ does not depend on $i,j,m,n$ and $r$. Set
\[
\Gam_r(m,n)=\cup_{u=1}^{i-1}U_r(q_u(m))\cup_{v=1}^jU_r(q_v(n))
\]
where $U_\rho(x)=\{ y:\, |x-y|\leq\rho\}$. Applying Proposition \ref{prop3.1} 
we conclude that
\begin{eqnarray}\label{4.6}
&|b_{i,j}^{(r)}(m,n)|=\big\vert E\big(E(Y_{i,r}(q_i(m))|\cF_{\Gam_r(m,n)})
Y_{j,r}(q_j(n))\big)\big\vert\\
&\leq\| Y_{j,r}(q_j(n))\|_2\| E(Y_{i,r}(q_i(m))|\cF_{\Gam_r(m,n)})\|_2\leq
Cr^\nu\vp_{q,p}(s_{i,j}(m,n)-2r).\nonumber
\end{eqnarray}
The estimate in the case $s_{i,j}(m,n)=\hat s_{j,i}(n,m)\geq 2r$ is similar.
Now we choose $r=\frac 14s_{i,j}(m,n)$ and take 
\begin{equation}\label{4.6+}
h(l)=C(l^\nu\vp_{p,q}(2l)+(\be_{p,q}(l))^\del)
\end{equation}
provided either $s_{i,j}(m,n)=\hat s_{i,j}(m,n)$ and $|m|\geq M$ or
$s_{i,j}(m,n)=\hat s_{j,i}(n,m)$ and $|n|\geq M$. Next, suppose that
$s_{i,j}(m,n)=\hat s_{i,j}(m,n)$ and $|m|< M$. Then $\hat s_{j,i}(n,m)\leq
s_{i,j}(m,n)<M$, and so either $|n|<M$ or $\min_{1\leq l\leq i}|q_j(n)-
q_l(m)|<M$. In the latter case we must have $|n|\leq N(M)$ for some $N(M)>0$
depending only on $M$. The same argument holds true when $s_{i,j}(m,n)=
\hat s_{j,i}(n,m)$ and $|n|<M$. Thus, when $l\geq M$ we can define $h(l)$ by
(\ref{4.6+}) while for $l<M$ we can take
\[
h(l)=\max\{ |b_{i,j}(m,n)|:\, 0\leq |m|,|n|\leq\max(M,N(M)),\, 1\leq i,j
\leq\ell\}
\]
concluding the proof of the lemma.
\end{proof}

For $s,t\in[0,1]^\nu$ with $s\leq t$ set
\[
\xi_{i,N}(s,t)=N^{-\nu/2}\sum_{n\in\Del_N(s,t)}F_i(X(q_1(n)),...,X(q_i(n))).
\]
Now we can obtain an appropriate estimate of the second moment of $\xi_{i,N}$.

\begin{lemma}\label{lem4.2} There exists $C>0$ such that for all $t=(t_1,...,
t_\nu)\geq s=(s_1,...,s_\nu)\geq 0$ and $i=1,...,\ell$,
\begin{equation}\label{4.7}
 E|\xi_{i,N}(s,t)|^2\le CN^{-\nu}|\Del_N(s,t)|\le C\prod_{l=1}^\nu (t_l-s_l
 +N^{-1}).
\end{equation}
\end{lemma}
\begin{proof}
Set $G(m,l)=\{ n\in\bbZ_+^\nu:\, l\leq |m-n|<l+1\}$. Then by (\ref{4.2}) and
(\ref{4.4}) for $i\leq k$,
\begin{eqnarray}\label{4.8}
&E|\xi_{i,N}(s,t)|^2= N^{-\nu}\sum_{m,n\in\Del_N(s,t)}b_{i,i}(m,n)\\
&\leq 2N^{-\nu}\sum^\infty_{l=0}\sum_{m\in\Del_N(s,t)}\sum_{n\in G(m,l)}
|b_{i,i}(m,n)|\nonumber\\
&\leq 2N^{-\nu}\sum^\infty_{l=0}h([\frac l{4k^2}])\sum_{m\in\Del_N(s,t)}
|G(m,l)|.
\nonumber
\end{eqnarray}
If $i>k$ then by (\ref{4.3}) and (\ref{4.4}),
\begin{eqnarray}\label{4.9}
&E|\xi_{i,N}(s,t)|^2\leq N^{-\nu}\sum_{m,n: |m|,|n|\leq M_1}|b_{i,i}(m,n)|\\
&+2N^{-\nu}\sum_{l=0}^\infty\sum_{m\in\Del_N(s,t)}\sum_{n\in G(m,l)}
|b_{i,i}(m,n)|\nonumber\\
&\leq N^{-\nu}(C_1+\sum_{l=0}^\infty h([l/2])\sum_{m\in\Del_N(s,t)}|G(m,l)|)
\nonumber\end{eqnarray}
for some $C_1>0$ independent of $s,t$ and $N$. Clearly, for any 
$l\geq 0$,
 \begin{equation}\label{4.10}
 |G(m,l)|\leq C_2(l+1)^{\nu-1}
 \end{equation}
 for some $C_2>0$ independent of $m$ and $l$ since it is bounded by the
 number of integer points between spheres of radius $l$ and $l+1$. Finally,
 by (\ref{4.8})--(\ref{4.10}) and the summability of $l^{\nu-1} h(l)$ we 
 derive (\ref{4.7}).
 \end{proof}

\begin{lemma}\label{lem4.3}  For each $t >0$ and $i,j\leq k$,
\[
\lim_{ u \to \infty} \limsup_{N\to\infty}N^{-\nu}\sum_{0\le n,n' \le Nt
\atop |in-jn'|\ge u}|b_{i,j}(n,n')|=0.
\]
\end{lemma}
\begin{proof}
Set $G_{i,j}(n,l)=\{ n'\in\bbZ^\nu_+:\, l\leq |in-jn'|<l+1\}$. By 
Lemma \ref{lem4.1},
\begin{eqnarray}\label{4.10+}
&N^{-\nu}\sum_{0\leq n,n'\leq Nt\atop |in-jn'|\ge u}|b_{i,j}(n,n')|\\
&\leq N^{-\nu}\sum_{u\leq l<\infty}\sum_{0\leq n\leq Nt}\sum_{n'\in G_{i,j}
(n,l)}|b_{i,j}(n,n')|\nonumber\\
&\leq N^{-\nu}\sum_{u\leq l<\infty}\sum_{0\leq n\leq Nt}
|G_{i,j}(n,l)|h(\frac l{4k^2})\leq C_3\sum_{u\leq l<\infty}(l+1)^{\nu-1}
h(\frac l{4k^2})\nonumber
\end{eqnarray}
for some constant $C_3>0$ independent of $u$ and $N$.
As $u\to\infty$ the right hand side of (\ref{4.10+}) tends to zero in view
of summability of $l^{\nu}h(l)$ in $l$.
\end{proof}

\begin{proposition}\label{prop4.4} For any $i,j\leq k$ and $s=(s_1,...,s_\nu),
t=(t_1,...,t_\nu)\geq 0$ the limit
\begin{equation}\label{4.11}
\lim_{N\to\infty}N^{-\nu}\sum_{0\le in \le Ns\atop 0\le jn' \le Nt,\,
in-jn'= u} b_{i,j}(n,n')=\frac {\up\prod_{l=1}^\nu\min(s_l,t_l)}{ij}\,c_{i,j}(u)
\end{equation}
exists for any $u\in\bbZ^\nu$ where $\up$ is the greatest common divisor of $i$
and $j$, $c_{i,j}(u)=0$ if $\up$ does not divide all components of $u\in\bbZ^\nu$
and $c_{i,j}(\up\tilde u)$ was defined by (\ref{2.19++}). Finally,
\begin{eqnarray}\label{4.13}
&\lim_{N\to\infty}E\xi_{i,N}(s)\xi_{j,N}(t)=
\lim_{N\to\infty}N^{-\nu}\sum_{0\le in\le Ns\atop 0\le jn'\le Nt}
b_{i,j}(n,n')\\
&=D_{i,j}\prod_{l=1}^\nu\min(s_l,t_l)\nonumber
\end{eqnarray}
with $D_{i,j}$ defined by (\ref{2.19+}).
\end{proposition}
\begin{proof} 
Let $\up$ be the greatest common divisor of $i$ and $j$. If
 $u\in\bbZ^\nu$ has components which are not divisible by $\up$ then the sum
 in (\ref{4.11}) is empty, and so in this case $c_{i,j}(u)=0$. Thus it
 remains to deal with this sum when $in-jn'=\up u$ for $u\in\bbZ^\nu$. 
 We will show first that the limit
\begin{equation}\label{4.15}
c_{i,j}(u)=\lim_{|n|,|n'|\to\infty,\, in-jn'=\up u}b_{i,j}(n,n')
\end{equation}
exists. Observe that if we consider two strings $(n, 2n \ldots, in)$ and
 $(n',2n',\ldots, jn')$ with $in-jn'=\up u$ then there will also be pairs
  $(i_\alpha, j_\alpha)$, $\alpha=1,2,\ldots, \up-1$ such that $i_\alpha
  n-j_\alpha\,n'=\alpha  u$ where $i_\al=\al i_1$ and $j_\al=\al j_1$
  with $i_1$ and $j_1$ being coprime. On the other hand, if $\tilde i/
  \tilde j\ne i_1/j_1$ then 
 \begin{eqnarray}\label{4.16}
 &|\tilde in-\tilde jn'|=\tilde j\big\vert\frac {\tilde i}{\tilde j}n-
 n'\big\vert =\tilde j\big\vert (\frac {\tilde i}{\tilde j}-
\frac {i_1}{j_1})n+\frac 1{j_1}(i_1n-j_1n')\big\vert\\
&\geq |n|\tilde j\big\vert\frac {\tilde i}{\tilde j}-\frac {i_1}{j_1}\big\vert 
-\frac {\tilde j}{j_1}|u|\rightarrow\infty\,\,\mbox{as}\,\, |n|\to\infty.
\nonumber\end{eqnarray} 

We split the collection of numbers $(n,2n,...,in;n',2n',...,jn')$ into 
disjoint sets $\Gam_1,...,\Gam_\up,...,\Gam_{i+j-\up}$ where $\Gam_\al=
\{\al i_1,\al j_1\},\,\al=1,...,\up$ are pairs and $\Gam_{\up+\be},\,\be=
1,2,...,i+j-\up$ are singeltons. We order the latter so that $\Gam_\be=
\{ l_\be n\},\, 1\leq l_\be <i,\,\be=\up+1,...,i$ with $l_\be\ne\al i_1$
for $\al=1,...,\up$ and $\Gam_{\up+\be}=\{{l'}_\be n'\}$ for $\be=i+1,...,
i+j-\up$, $1\leq {l'}_\be<i,\, {l'}_\be\ne\al j_1$ for $\al=1,...,\up$.
By (\ref{4.16}) there is $\del>0$ depending on $u$ but not on $n$ and $n'$
such that
\[
\min_{1\leq l\ne l'\leq i+j-\up}\mbox{dist}(\Gam_l,\Gam_{l'})\geq\del n.
\]
Set
\[
U_r(\Gam_l)=\{ n\in\bbZ^\nu_+:\,\mbox{dist}(n,\Gam_l)\leq r\},\,\, l=1,2,...,
i+j-\up
\]
and choose $r=r(n)\to\infty$ as $|n|\to\infty$ so that all $U_{r(n)}(\Gam_l),\,
l=1,2,...,i+j-\up$ were disjoint.

Now, observe that $b_{i,j}(n,n')$ has the form $EG(Y_1(n,n'),Y_2(n,n'),...,
Y_{i+j-\up}(n,n'))$ where $Y_\al(n,n')=(X(\al i,n),X(\al j,n'))$ for 
$\al=1,...,\up$, $Y_\be(n,n')=X(l_\be n)$ for $\be=\up+1,...,i$ and
$Y_\be(n,n')=X(l^\prime_\be n')$ for $\be=i+1,...,i+j-\up$. Define $G_1=G$ 
and successively
\[
G_{l+1}(y_{l+1},y_{l+2},...,y_{i+j-\up})=EG(Y_l(n,n'),y_{l+1},y_{l+2},...,
y_{i+j-\up}).
\]
Relying on the assumptions (\ref{2.11}) and (\ref{2.12}) we can apply 
(\ref{3.3}) of Proposition \ref{prop3.1} with $V=(Y_{l+1}(n,n'),Y_{l+2}(n,n'),...,
Y_{i+j-\up}(n,n'))$, $Z=Y_l(n,n')$, $\cG=\cF_{\tilde U}$ with $\tilde U=
\cup_{l+1\leq\be\leq i+j-\up}U_{r(n)}(\Gam_\be)$ and $\cH=\cF_{U_l(r(n))}$ 
obtaining that
\begin{eqnarray*}
&EG_l(Y_l(n,n'), Y_{l+1}(n,n'),...,Y_{i+j-\up}(n,n'))\\
&-EG_{l+1}(Y_{l+1}
(Y_{l+1}(n,n'),...,Y_{i+j-\up}(n,n'))\to 0\,\,\mbox{as}\,\, n\to\infty.
\end{eqnarray*}
This argument repeated for $l=1,2,...,i+j-\up-1$ yields (\ref{4.15})
with $c_{i,j}$ given by (\ref{2.19++}).

Finally, in order to obtain (\ref{4.11}) and (\ref{4.13}) we have to 
count the number of solutions $n,n'$ of the vector Diophantine equation
$in-jn'=\up u$ where $0\leq in\leq Ns$ and $0\leq jn'\leq Nt$. Since we
have to satisfy this equation coordinate wise, the number of solutions is
the product of the number of solutions in each coordinate. 
Let, as before, $i=\up i_1$ and $j=\up j_1$ with $i_1$ and $j_1$ being 
coprime then all solutions
of the equation $i_1n_l-j_1{n'}_l=u_l$ are given by $n=n_0+j_1m$ and $n'=
n'_0+i_1m$ where $n_0,n'_0$ is its particular solution and $m$ is any integer.
The number of such solutions with $0\leq in_l\leq Ns_l$ and $0\le j{n'}_l
\le Nt_l$ for large $N$ is equal approximately to
\[
N\min(\frac {t_l}{ij_1},\frac {s_l}{ji_1})+O(1)=\frac {N\up\min(s_l,t_l)}{ij}+O(1)
\]
and taking the product in $l$ we obtain (\ref{4.11}) while (\ref{4.13}) follows
from (\ref{4.11}) and Lemmas \ref{lem4.1} and \ref{lem4.3}.
\end{proof}

\begin{proposition}\label{prop4.5} For $i\ge k+1$, 
\begin{equation}\label{4.17}
\lim_{N\to\infty}E\big |\xi_{i,N}(t)|^2=(\prod_{l=1}^\nu t_l)\int\big(F_i(x_1, 
x_2,\ldots, x_i)\big)^2d\mu(x_1)d\mu(x_2)\cdots d\mu(x_i).
\end{equation}
Moreover, for any $t,s\in\bbR^\nu_+$ and $j<i$,
\begin{equation}\label{4.18}
\lim_{N\to\infty}E\big (\xi_{i,N}(t)\xi_{j,N}(s)\big )=0.
\end{equation}
\end{proposition}
\begin{proof} By (\ref{2.8})--(\ref{2.9}), (\ref{4.3}) and (\ref{4.4}),
\[
b_{i,i}(m,n)\to 0\,\,\mbox{as}\,\,\max(|m|,|n|)\to\infty\,\,\mbox{with}\,\,
|m-n|\geq 1.
\]
Therefore, by (\ref{4.3}) and (\ref{4.4}) for any fixed $L$,
\begin{eqnarray*}
&\limsup_{N\to\infty}N^{-\nu}\sum_{m,n\in\Del_N(t),\, m\ne n}|b_{i,i}(m,n)|\\
&\leq\limsup_{N\to\infty}N^{-\nu}\big(\sum_{1\leq |m-n|\leq L}|b_{i,i}(m,n)|\\
&+C|\Del_N(t)|\sum_{l\geq L}l^{\nu-1}h(l)\big)=\tilde C
\sum_{l\geq L/2}l^{\nu-1}h(l)
\end{eqnarray*}
for some $C,\tilde C>0$ independent of $N$ and $L$. We now let $L\to\infty$ and since 
 $l^{\nu-1}h(l)$ is summable it follows that $\limsup$ in the left hand side
above equals zero, i.e. the off-diagonal terms do not contribute in 
(\ref{4.17}). It remains to deal with the diagonal terms $b_{i,i}(n,n)$.
Since $|q_j(n)-q_{j-1}(n)|\to\infty$ for $j=2,3,...,\ell$ as $|n|\to\infty$
it follows by the argument similar to one applied in the proof of Proposition
\ref{prop4.4} (see a more general Lemma 4.3 in \cite{KV}) that 
\begin{equation}\label{4.20}
\lim_{|n|\to\infty}b_{i,i}(n,n)=\int\big(F_i(x_1,x_2,...,x_i)\big)^2d\mu(x_1)
\cdots d\mu(x_i).
\end{equation}
 Namely, set
$G_i(x_1,...,x_i)=\big(F_i(x_1,x_2,...,x_i)\big)^2$ and recursively for
$l=i-1,...,2,1,0$, 
\[
G_l(x_1,...,x_l)=\int G_{l+1}(x_1,...,x_{l+1})d\mu(x_{l+1}).
\]
Taking into account that $|q_l(n)-q_{\tilde l}(n)|\to\infty$ as $|n|\to\infty$ 
when $l\ne\tilde l$ we apply Proposition \ref{prop3.1} to obtain successively 
for $l=i,i-1,...,1,0$ that
\[
\big\vert EG_{l+1}\big(X(q_1(n)),...,X(q_l(n))\big)-EG_{l+1}\big(X(q_1(n)),...,
X(q_l(n))\big)\big\vert\to 0\,\,\mbox{as}\,\, |n|\to\infty.
\]
Since $b_{i,i}(n,n)=G_0$ we arrive at (\ref{4.20}).

Next, we deal with (\ref{4.18}). Since $i>j$ and $i>k$ then by 
(\ref{2.8})--(\ref{2.9}) for any $\ve>0$ there exists $N(\ve)$ such that
 whenever $|m|>\ve N$ and $|n|\leq N\sqrt\nu$ we have $s_{i,j}(m,n)
 \geq\min(|m-n|+\ve^{-1},|m|)$ provided $N\geq N(\ve)$ where $s_{i,j}(m,n)$ is the same
 as in Lemma \ref{lem4.1}. It follows from Lemmas \ref{lem4.1} and 
 \ref{lem4.2} that
 \begin{eqnarray}\label{4.21}
 &|E\xi_{i,N}(t)\xi_{j,N}(s)|\leq \big\vert\sum_{|m|\leq\ve N,\, n\in\Del_N(s)}
 b_{i,j}(m,n)\big\vert\\
 & +N^{-\nu}\sum_{|m|>\ve N,\, m\in\Del_N(t),n\in\Del_N(s)}
 |b_{i,j}(m,n)|\nonumber\\
 &\leq N^{-\nu}\big(\sum_{|m|\leq\ve N}EY_{i,m}^2\big)^{1/2}
 \big(\sum_{n\in\Del_N(s)}EY^2_{j,n}\big)^{1/2}\nonumber\\
 &+C\sum_{l\geq\min(\ve^{-1},\ve N)}l^{\nu-1} h(l)
 \leq C\sqrt{\ve}+C\sum_{l\geq\min(\ve^{-1},\ve N)}l^{\nu-1} h(l)\nonumber
 \end{eqnarray}
 for some $C>0$ independent of $N$ and $\ve$. Letting in (\ref{4.21}), 
 first $N\to\infty$ and then $\ve\to 0$ we arrive at (\ref{4.18}).
\end{proof}

\section{Tightness estimates}\label{sec5}\setcounter{equation}{0}

First, we will extend the estimate of Lemma \ref{lem4.2} to the corresponding
estimate of the 4th moment.
\begin{lemma}\label{lem5.1} There exists $C>0$ such that for all $t=(t_1,...,
t_\nu)\geq s=(s_1,...,s_\nu)\geq 0$ and $i=1,...,\ell$,
\begin{equation}\label{5.1}
 E|\xi_{i,N}(s,t)|^4\le CN^{-2\nu}|\Del_N(s,t)|^2\le C\prod_{l=1}^\nu (t_l-s_l
 +N^{-1})^2.
\end{equation}
\end{lemma}
\begin{proof}
For $n^{(1)},n^{(2)},n^{(3)},n^{(4)}\in\bbZ^\nu_+$ set 
\[
d_i(n^{(1)},n^{(2)},n^{(3)},n^{(4)})=EY_{i,q_i(n^{(1)})}Y_{i,q_i(n^{(2)})}
Y_{i,q_i(n^{(3)})}Y_{i,q_i(n^{(4)})}
\]
and for $r>0$,
\[
d_i^{(r)}(n^{(1)},n^{(2)},n^{(3)},n^{(4)})=EY_{i,q_i(n^{(1)}),r}
Y_{i,q_i(n^{(2)}),r}Y_{i,q_i(n^{(3)}),r}Y_{i,q_i(n^{(4)}),r}.
\]
Then similarly to (\ref{4.5}),
\begin{equation}\label{5.2}
|d_i^{(r)}(n^{(1)},n^{(2)},n^{(3)},n^{(4)})-d_i(n^{(1)},n^{(2)},n^{(3)},
n^{(4)})|\leq C_1(\be_q(r))^\del
\end{equation}
where $C_1>0$ does not depend on $n^{(1)},n^{(2)},n^{(3)},n^{(4)}$ and $r$.
Define
\begin{eqnarray*}
&v_i(n^{(1)},n^{(2)},n^{(3)},n^{(4)})=\max_{1\leq j\leq 4}\big(\min
\big(\min_{l<i}|q_i(n^{(j)})-q_l(n^{(j)})|,\\
&\min_{\tilde j\ne j,l\leq i}|q_i(n^{(j)})-q_l(n^{(\tilde j)})|\big)\big).
\end{eqnarray*}
Without loss of generality assume that
\begin{equation}\label{5.3}
v_i(n^{(1)},n^{(2)},n^{(3)},n^{(4)})=\min\big(\min_{l<i}|q_i(n^{(j)})-
q_l(n^{(j)})|,\min_{\tilde j\ne j,l\leq i}|q_i(n^{(j)})-q_l(n^{(\tilde j)})|
\big).
\end{equation}
For each $a\geq 0$ introduce the sets
\[
\Gam_a=\{ n^{(1)},n^{(2)},n^{(3)},n^{(4)}\in\bbZ_+:\, a\leq v_i(n^{(1)},
n^{(2)},n^{(3)},n^{(4)})<a+1\}
\]
and
\[
\Gam_a(N,s,t)=\{ (n^{(1)},n^{(2)},n^{(3)},n^{(4)})\in\Gam_a:\, 
n^{(1)},n^{(2)},n^{(3)},n^{(4)}\in\Del_N(s,t)\}.
\]

If $i\leq k$ and $(n^{(1)},n^{(2)},n^{(3)},n^{(4)})\in\Gam_a$ then for 
$j=2,3,4$,
\[
\mbox{either}\,\, |n^{(j)})<a+1\,\,\mbox{or}\,\, |in^{(j)}-ln^{(\tilde j)}|
<a+1\,\,\mbox{for some}\,\, l=1,...,i\,\,\mbox{and}\,\,\tilde j\ne j.
\]
It follows that
\begin{equation}\label{5.4}
|\Gam_a(N,s,t)|\leq C_2a^2(1+a^2+|\Del_N(s,t)|^2)
\end{equation}
for some $C_2>0$ independent of $a,N,s$ and $t$. If $i\geq k+1$ then by
(\ref{2.8})--(\ref{2.9}) there exists $M>0$ such that whenever $|n|\geq M$,
\begin{equation}\label{5.5}
\min_{l<i}|q_i(n)-q_l(n)|\geq |n|\,\,\mbox{and}\,\,\min_{\tilde i\leq i,\,
\tilde n\ne n}|q_i(n)-q_{\tilde i}(\tilde n)|\geq |n-\tilde n|.
\end{equation}
Then similarly to the case $i\leq k$ we conclude from (\ref{5.5}) that 
(\ref{5.4}) holds true also when $i\geq k+1$.

Next, let $r=a/3$ and $(n^{(1)},n^{(2)},n^{(3)},n^{(4)})\in\Gam_a$ satisfy
(\ref{5.3}). Set
\[
\Psi_r(n^{(1)},n^{(2)},n^{(3)},n^{(4)})=\cup_{u=1}^{i-1}U_r(q_u(n^{(1)}))
\cup\big(\cup_{j=2}^4\cup_{v=1}^iU_r(q_v(n^{(j)})\big).
\]
Then by Proposition \ref{prop3.1} similarly to (\ref{4.6}) we derive that
\begin{eqnarray}\label{5.6}
&|d_i^{(r)}(n^{(1)},n^{(2)},n^{(3)},n^{(4)})|\leq\big\vert E\big( E(Y_{i,
q_i(n^{(1)}),r}|\cF_{\Psi_r(n^{(1)},n^{(2)},n^{(3)},n^{(4)})})\\
&\times Y_{i,q_i(n^{(2)}),r}Y_{i,q_i(n^{(3)}),r}Y_{i,q_i(n^{(4)}),r}\big)
\big\vert\leq C_3a^\nu\vp_{q,p}(a/3)\nonumber
\end{eqnarray}
for some $C_3>0$ independent of $a,n^{(1)},n^{(2)},n^{(3)}$ and $n^{(4)}$.
Set $q(a)=C_1(\be_q(a/3))^\del+C_3a^\nu\vp_{q,p}(a/3)$. Then
\begin{eqnarray}\label{5.7}
&\quad E|\xi_{i,N}(s,t)|^4\leq\sum_{a=0}^\infty\sum_{(n^{(1)},n^{(2)},n^{(3)},
n^{(4)})\in\Gam_a(N,s,t)}|d_i^{(r)}(n^{(1)},n^{(2)},n^{(3)},n^{(4)})|\\
&\leq C_2\sum_{a=0}^\infty q(a)a^2(1+a^2+|\Del_N(s,t)|^2)\nonumber
\end{eqnarray}
and (\ref{5.1}) follows from (\ref{5.7}) and Assumption \ref{ass2.1}.
\end{proof}

Now tightness of each sequence of random fields $\{ \xi_{i,N}(t),\, 
t\in[0,1]^\nu\}$ follows by a slight modification of \cite{BW} (see also
Ch.5 in \cite{BS2} and Theorem 1.4.7 in \cite{Ku}), and so
the sequence of random fields $\{ \xi_N(t),\, t\in[0,1]^\nu\}$ is tight,
as well.

\section{Gaussian limits}\label{sec6}\setcounter{equation}{0}

For each fixed $t\in[0,1]^\nu$ the convergence in distribution as 
$N\to\infty$ of each $\zeta_{i,N}(t),\, i=1,...,\ell$ to corresponding
Gaussian (maybe degenerated) random variables follows from central limit 
theorems for mixingale arrays (see \cite{ML2}, \cite{Jo} and references there)
in view of the mixingale estimates of Section \ref{sec3}, convergence of
covariances obtained in Section \ref{sec4} and the tightness result derived
in Section \ref{sec5}.
Then $\xi_{i,N}(t),\, i=1,...,\ell$ also converge in distribution to the same
Gaussian random variables in view of (\ref{3.23}). Furthermore, for any
$\bt=(t^{(1)},...,t^{(j)}),\, t^{(a)}\in[0,1]^\nu,\, a=1,...,j$ and
$\bd=(d_1,...,d_j)$ set
\begin{equation}\label{6.1}
V_{i,\bt,\bd,N}(l)=\sum_{a=1}^j d_aV_{i,t^{(a)},N}(l)
\end{equation}
and
\begin{equation}\label{6.2}
\zeta_{i,\bd,N}(\bt)=\sum_{a=1}^j d_a\zeta_{i,N}(t^{(a)})=N^{-\nu/2}
\sum_{1\leq l\leq L(N)}V_{i,\bt,\bd,N}(l).
\end{equation}
Then, we obtain from (\ref{3.13}) and (\ref{3.14}) similar mixingale 
estimates also for $V_{i,\bt,\bd,N}(l)$ which via \cite{ML2} and
\cite{Jo} yields convergence in distribution as $N\to\infty$ to Gaussian
random variables of each $\zeta_{i,\bd,N}(\bt)$. This together with
(\ref{3.23}) imply that each 
\begin{equation}\label{6.3}
\xi_{i,\bd,N}(\bt)=\sum_{a=1}^j d_a\xi_{i,N}(t^{(a)})
\end{equation}
converges in distribution as $N\to\infty$ to Gaussian random variables.
Hence, finite dimensional distributions of each $\xi_{i,N}$ have Gaussian
limits which together with tightness results of Section \ref{sec5} yields
 that each $\xi_{i,N}$ converges in distribution as $N\to\infty$ to a
 Gaussian random field $\eta_i$.
 
 In fact, we can show that $(\xi_{1,N},...,\xi_{1,k})$ converges in 
 distribution as $N\to\infty$ to a $k$-dimensional Gaussian random 
 field $(\eta_1,...,\eta_k)$. Indeed, for any $\bfe=(e_1,...,e_k)\in\bbR^k$
 set
 \begin{equation}\label{6.4}
 V_{\bt,\bd,\bfe,N}(l)=\sum_{i=1}^ke_iV_{i,\bt,\bd,N}(l)\,\,\mbox{and}\,\,
 \zeta_{\bd,\bfe,N}(\bt)=\sum_{i=1}^ke_i\zeta_{i,\bd,N}(\bt).
 \end{equation}
 Then it is easy to see again by (\ref{3.13}) and (\ref{3.14}) that similar
 mixingale estimates hold true also for $V_{\bt,\bd,\bfe,N}(l)$ which via
 \cite{ML2} and \cite{Jo} provides convergence in distribution as $N\to
 \infty$ of $\zeta_{\bd,\bfe,N}$ to a Gaussian random variable which must
 have the same distribution as $\sum_{i=1}^k\sum_{a=1}^je_i\eta_i(t^{(a)})$.
 As above we conclude from (\ref{3.23}) and tightness arguments of Section
 \ref{sec5} that, in fact, $\sum_{i=1}^ke_i\xi_{i,N}$ converges in distribution
 as $N\to\infty$ to a Gaussian random field which must have the same 
 distribution as $\sum_{i=1}^ke_i\eta_i$. Thus, $(\eta_1,...,\eta_k)$ is a
 Gaussian random field and $(\xi_{1,N},...,\xi_{k,N})$ converges in
 distribution to it as $N\to\infty$. Finally, $\sum_{i=1}^k\xi_{i,N}(it)$
 converges in distribution as $N\to\infty$ to the random field
 $\sum_{i=1}^k\eta_i(it)$ which must be Gaussian as a result of the linear
  transformation (in the path space) of a Gaussian random field (see, for
  instance, \cite{Bog}, Section 2.2).
  
  Next, clearly, $\xi_N$ converges in distribution to $\xi$ given by
  (\ref{2.20}) and it remains to show that $\eta_i$ with $i\geq k+1$ are
  independent of each other and of $\eta_i$ with $i\leq k$ which will 
  imply that $\xi$ is a Gaussian random field. This can be done either
  via a modified version of Theorem 5.6 from \cite{KV} or by the following
  more direct approach. First, observe that (\ref{2.9}) implies  that
  there exists $\ve_N\to 0$ as $N\to\infty$ such that
  \begin{equation}\label{6.5}
  \lim_{N\to\infty}\min_{n,\tilde n\in\Del_N({\bf 1})\setminus\Del_N
  (\ve_N{\bf 1})}(|q_i(\tilde n)|-\max_{l<i}|q_l(n)|-|\tilde n|)=\infty
  \end{equation}
  where ${\bf 1}=(1,...,1)\in[0,1]^\nu$. We see from (\ref{6.5}) that if
  $n,\tilde n\in\Del(\ve_n\bf 1,t)$ then any $q_i(n),\, i\geq k+1$ and
  $q_j(\tilde n),\, j\ne i$ are widely separated. Next, it follows from 
  Lemma \ref{lem4.2} that for each $i$,
  \begin{equation}\label{6.6}
  \xi_{i,N}(\ve_N{\bf 1})\to 0\,\,\mbox{in probability as}\,\, 
  N\to\infty,
  \end{equation}
  and so in all our arguments the sum over $\Del_N(\ve_N{\bf 1})$ can be
   disregarded. 
  
  In order to use (\ref{6.5}) and (\ref{6.6}) introduce 
  $j_N=\min\{ j:\, a(j)\geq\ve_NN\}$ with $a(j)$ defined in (\ref{3.6}).
  For $\bt=(t^{(1)},...,t^{(j)}),\, t^{(a)}\in[\ve_N,1]^\nu,\, a=1,...,j$,
  $\bd=(d_1,...,d_j)$ and $\bfe=(e_{k+1},...,e_\ell)$ define for $l=1,2,...,
  L(N)$,
  \begin{equation}\label{6.7}
  \tilde V_{\bt,\bd,{\bfe},N}(l)=\sum_{a=1}^j\sum_{i=k+1}^\ell 
  d_ae_iV_{i,t^{(a)},N}(l)
  \end{equation}
  as in (\ref{6.1}) and (\ref{6.4}). On the other hand, for $(i-k)L(N)+j_N
  \leq l\leq (i+1-k)L(N)$ and $i\geq k+1$ we set
  \begin{equation}\label{6.8}
  \tilde V_{\bt,\bd,{\bfe},N}(l)=\sum_{a=1}^j d_ae_iV_{i,t^{(a)},N}
  (l-(i-k)L(N)
  \end{equation}
  setting $\tilde V_{\bt,\bd,\bfe,N}(l)$ to be zero if it is not defined 
  by one of the above and assuming that $t^{(a)}>\ve_N$ for all $a$. 
  It is easy to see from (\ref{2.8})--(\ref{2.9}), (\ref{6.5}), (\ref{6.6})
  and mixingale estimates of Section \ref{sec3} that the sequence  
  $(\tilde V_{\bt,\bd,\bfe,N}(l)-E\tilde V_{\bt,\bd,\bfe,N}(l))N^{-\nu/2},\,
  l=j_N,...,\ell L(N)$ forms a mixingale 
  array with the corresponding $\sig$-algebras $\cG_l=\cG_l^{(i)}$ for 
  $i\leq k$, $\cG_{l+(i-k)L(N)}=\cG_l^{(i)}$ for $i\geq k+1$ and estimates
  similar to (\ref{3.13+++}). Observe that
  \[
  \tilde\zeta_{\bd,{\bfe},N}(\bt)=N^{\nu/2}\sum_{j_N\leq l\leq L(N)}
  \sum_{i=1}^\ell e_iV_{i,\bt,
  \bd,N}(l)=N^{-\nu/2}\sum_{1\leq l\leq\ell L(N)}\tilde V_{\bt,\bd,\bfe,N}(l).
  \]
  Taking into account (\ref{6.6}), mixingale estimates of Section \ref{sec3}
  and convergence of covariances results of Section \ref{sec4}
  and tightness arguments of Section \ref{sec5}
  we conclude by mixingale limit theorems from \cite{ML2} and \cite{Jo} that
  $\tilde\zeta_{\bd,{\bfe},N}(\bt)-E\tilde\zeta_{\bd,{\bfe},N}(\bt)$ converges
  in distribution as $N\to\infty$ to a Gaussian random variable while similarly
  to the end of Section \ref{sec3} we see that 
  $\tilde\zeta_{\bd,{\bfe},N}(\bt)$ converges to the same limit.
  It follows from here in view of (\ref{3.23}) and (\ref{6.6})
  that each linear combination $\sum_{a=1}^j\sum_{i=1}^\ell 
  d_ae_i\xi_{i,N}(t^{(a)})$ converges in distribution as $N\to\infty$ to 
  some Gaussian random variable which then must be $\sum_{a=1}^j\sum_{i=1}^\ell 
  d_ae_i\eta_i(t^{(a)})$. Thus $(\eta_1,\eta_2,...,\eta_\ell)$ is an
  $\ell$-dimensional Gaussian random field. Invoking again the linear 
  transformation argument from Section 2.2 of \cite{Bog} we conclude 
   that the field $\xi(t)$ given by (\ref{2.20}) is a Gaussian one
  and taking into account the vanishing limiting covariances result
  (\ref{4.18}) we obtain also independency of $\eta_i,\, i\geq k+1$ of
  each other and of $\eta_i$'s with $i\leq k$.
  \qed
  
 \bibliography{matz_nonarticles,matz_articles}
\bibliographystyle{alpha}

\end{document}